\documentclass[12pt, reqno]{amsart}
\usepackage{amsmath, amsthm, amscd, amsfonts, amssymb, graphicx, color}
\usepackage[bookmarksnumbered, colorlinks, plainpages]{hyperref}

\makeatletter \oddsidemargin.9375in \evensidemargin \oddsidemargin
\newtheorem{theorem}{Theorem}[section]
\usepackage{titlesec}
\titleformat{\section}{\bfseries\Large}{\thesection}{1em}{}
\titleformat{\subsection}{\bfseries\normalsize}{\thesubsection}{1em}{}

\theoremstyle{definition}

\numberwithin{equation}{section}
\usepackage{amsthm}
\newtheoremstyle{boldquestion}  
{\topsep}                      
{\topsep}                      
{\normalfont}                 
{}                            
{\bfseries}                   
{.}                           
{ }                           
{\thmname{#1}\thmnumber{ #2}} 

\theoremstyle{boldquestion}
\newtheorem{question}{Question}
\usepackage{amsthm}

\newtheoremstyle{boldremark}  
{}{}                     
{\normalfont}              
{}                         
{\bfseries}                
{.}                        
{ }                       
{}

\theoremstyle{boldremark}
\newtheorem{remark}{Remark}

\usepackage{amsthm}

\makeatletter
\renewenvironment{proof}[1][\proofname]{%
	\par\pushQED{\qed}\normalfont
	\topsep6\p@\@plus6\p@\relax
	\trivlist
	\item[\hskip\labelsep\bfseries #1\@addpunct{.}]%
}{%
	\popQED\endtrivlist\@endpefalse
}
\makeatother

\usepackage{bm}

\begin{document}
\setcounter{page}{1}


\title[]{Concerning FAT Colorings of Graphs}
\author[]{Saeed Shaebani}
\thanks{{\scriptsize
\hskip -0.4 true cm MSC(2020): Primary: 05C15; Secondary: 05C07.
 }}

\begin{abstract}
\noindent Let $G$ be a graph and let $C$ be a color set of cardinality $k$. Suppose $c \colon V(G) \to C$ is a (not necessarily proper) vertex coloring whose all color classes are $V_1$, $V_2$, $\dots$, $V_k$, each of which is nonempty. The vertex coloring $c$ is said to be a  {\it FAT $k$-coloring of $G$} if there exist real numbers $\alpha$ and $\beta$, both in $[0,1]$, such that for every vertex $v\in V(G)$ and every color class $V_i$ the following equalities hold:
$$
\bigl| V_i \cap N(v) \bigr|  =
\begin{cases}
	\alpha \deg (v)  &  \mbox{ if }  \ \   v \notin V_i \\
	\beta \deg (v)   &  \mbox{ if }  \ \   v \in V_i . 
\end{cases}
$$
Let $k > 1$ be a fixed integer, and let $\alpha \in \left[ 0 , \frac{1}{k-1} \right) \cap \mathbb{Q}$
and $\beta \in [ 0 , 1 ] \cap \mathbb{Q}$ be
some fixed rational numbers satisfying $ \beta + (k-1) \alpha = 1 $.
It was asked for the existence of a graph $G$ with $\delta (G) > 0$ admitting some FAT $k$-coloring
with the corresponding parameters $\alpha$ and $\beta$.
This paper settles the question in the affirmative.
We explicitly construct a sequence 
$\displaystyle\left\{G_n\right\}_{n=1}^{\infty}$ of pairwise non-homomorphically equivalent graphs, 
each being a regular graph of positive degree, admitting a FAT $k$-coloring with the corresponding 
parameters $\alpha$ and $\beta$.
\end{abstract}

\maketitle
\section{Introduction}

This article is devoted to the study of graph coloring parameters. Except where $[\,\cdot\,]$ denotes bibliographic citations, the notation $[n]$ where $n$ is a positive integer, will be used to represent the set $\{1,2,\ldots,n\}$.

The present study is confined to simple graphs whose vertex sets are finite and nonempty. We adopt conventional graph-theoretic notation throughout, in accordance with the standard treatments found in \cite{Bollobas1998}, \cite{BondyMurty2008}, \cite{Diestel2017}, and \cite{West2001}.

The {\it order} of a graph is the number of its vertices, and the {\it size} of that graph is the number of its edges.
Let $G$ be a graph and $v$ a vertex of $G$. The {\it neighborhood} of $v$, denoted $N_G(v)$, is the collection of vertices adjacent to $v$; each element of $N_G(v)$ is referred to as a
{\it neighbor} of $v$. The {\it degree} of $v$ is the cardinality $|N_G(v)|$, and is written $\deg_G(v)$. When the ambient graph $G$ is evident from context, we shall abbreviate $N_G(v)$ and $\deg_G(v)$ simply as $N(v)$ and $\deg(v)$, respectively.

\noindent
By a {\it regular} graph we mean a graph in which all vertices have the same degree.

For a graph $G$ and every subset $S$ of $V(G)$, the {\it induced subgraph of $G$ on $S$}, denoted by $G[S]$, is a graph with vertex set $V(G[S]) := S$ whose edge set equals
$$E(G[S]) := \big\{xy \colon \ \ x \in S  \mbox{ and } y \in S \mbox{ and } xy \in E(G) \big\} .$$

By a {\it clique} in a graph $G$, we mean a subset $S \subseteq V(G)$ in such a way that $G[S]$ is a complete graph, that is, every two pairwise distinct vertices in $S$ are adjacent.
The {\it clique number} of $G$, written $\omega (G)$, stands for the maximum cardinality of a clique in $G$.

Let $G$ be a graph. A {\it vertex coloring} of $G$ is a function $c\colon V(G)\to C$, where $C$ is designated the {\it color set} and each element of $C$ is called a {\it color}. For a vertex $v\in V(G)$, the equality $c(v)=r$ indicates that $v$ is assigned the color $r$.

\noindent
For any color $r\in C$, the {\it color class} of $r$ is the subset $c^{-1}(r)=\{v\in V(G) \colon \ c(v)=r\}$; equivalently, it consists precisely of those vertices that are assigned the color $r$.

\noindent
The collection of all nonempty color classes induced by a coloring
$c \colon V(G) \to C$ constitutes a partition of $V(G)$. Independent of the particular labels assigned to colors, this partition decomposes $V(G)$ into pairwise disjoint blocks, each block consisting of those vertices that share the same color. Conversely, any partition of $V(G)$ may be realized as the family of nonempty color classes produced by a suitable vertex coloring of $G$.

For a graph $G$, let $v\in V(G)$ be an arbitrary vertex and let $S\subseteq V(G)$ be any subset. Define $e(v,S)$ to be the number of neighbors of $v$ that lie in $S$. More precisely,
$$e (v,S) := \bigl| S \cap N(v) \bigr| .$$

\noindent
A coloring $c \colon V(G) \to C$ is {\it proper} if every pair of adjacent vertices are assigned distinct colors.
The minimum cardinality of a set $C$ for which the graph $G$ admits a proper coloring $c \colon V(G) \to C$ is called the {\it chromatic number} of $G$ and it is denoted by $\chi (G)$.

Let $G$ be a graph and let $C$ be a color set of cardinality $k$. Suppose $c \colon V(G) \to C$ is a (not necessarily proper) vertex coloring whose all color classes are $V_1$, $V_2$, $\dots$, $V_k$, each of which is nonempty. The vertex coloring $c$ is said to be a {\it Fair And Tolerant vertex coloring of $G$ with $k$ colors} if there exist real numbers $\alpha$ and $\beta$, both in $[0,1]$, such that for every vertex $v\in V(G)$ and every color class $V_i$ the following equalities hold:
$$
e \left( v, V_i \right) =
\begin{cases}
	\alpha \deg (v)  &  \mbox{ if }  \ \   v \notin V_i \\
	\beta \deg (v)   &  \mbox{ if }  \ \   v \in V_i . 
\end{cases}
$$

\medskip

\noindent
The notion of Fair And Tolerant vertex colorings was introduced and examined by Beers and Mulas in \cite{beers20252}. In that work they adopted the concise label “{\it FAT Colorings}”, and they likewise used the term “{\it FAT $k$-coloring}” to denote a Fair And Tolerant vertex coloring employing $k$ colors.

\medskip

Let $G$ be a graph containing at least one vertex $v$ with $\deg(v)>0$. Assume $G$ admits a FAT $k$-coloring whose nonempty color classes are $V_1$, $V_2$, $\dots$, $V_k$, and let $\alpha,\beta\in[0,1]$ denote the associated parameters. For such a vertex $v$ we obtain
$$\deg (v) = \displaystyle\sum _{i=1} ^k  e \bigl( v , V_i\bigr) =	\beta \deg (v) + (k-1) \alpha \deg (v) ,$$
whence the relation
$$
\beta+(k-1)\alpha=1
$$
follows, establishing a constraint among $k$, $\alpha$, and $\beta$ as observed in \cite{beers20252}.

\medskip

Let $G$ be a graph. On one hand, a canonical FAT $1$-coloring of $G$ is obtained by taking the sole color class to be $V(G)$ and setting $\alpha=0$ and $\beta=1$. Under this assignment, for each vertex $v\in V(G)$ we have $e(v,V(G))=\deg(v)=\beta \deg(v)$, so the defining relations of a FAT $1$-coloring are satisfied. On the other hand, in every FAT coloring of $G$, each color class is nonempty and the color classes form a partition of $V(G)$; so their number cannot exceed the cardinality of $V(G)$.
Accordingly, the set of all positive integers $k$
for which some FAT $k$-coloring of $G$ exists, is nonempty and bounded. Beers and Mulas \cite{beers20252} designated the maximum element of this set as the {\it FAT chromatic number} of $G$,
and they also represented it by $\chi ^{{\rm FAT}} (G)$. More precisely, the FAT chromatic number of $G$, denoted  $\chi ^{{\rm FAT}} (G)$, is the largest positive integer $k$
for which $G$
admits a FAT $k$-coloring \cite{beers20252}.

\section{The Main Result}

Let us regard a positive integer $k > 1$ as fixed. As discussed earlier in the previous section, whenever some graph $G$ with $E (G) \neq \varnothing$
admits a FAT $k$-coloring with associated parameters $\alpha$ and $\beta$,
then the relation
$$ \beta + (k-1) \alpha = 1 $$
is satisfied \cite{beers20252}.
Since $\alpha$ and $\beta$ are assumed to be nonnegative, and $k-1$ is positive, one finds that
$$\alpha \leq \frac{1}{k-1} .$$ 
On the other hand, it was shown by Beers and Mulas \cite{beers20252} that the complete graph $K_k$ on vertex set
$\left\{ x_1 , x_2 , \dots , x_{k} \right\}$ admits a FAT $k$-coloring
with corresponding parameters $\alpha = \frac{1}{k-1}$ and $\beta = 0$ whose color classes are
the singletons $\left\{ x_1 \right\}$, $\left\{ x_2 \right\}$, $\dots$, $\left\{ x_{k} \right\}$.
So, the following interesting question was raised in \cite{beers20252}.

\medskip

\begin{question} {\rm (Beers and Mulas \cite{beers20252})}  \label{Q7} \\
	Let $k > 1$ be a fixed integer and $\alpha \in \left[ 0 , \frac{1}{k-1} \right)$ and $\beta \in [ 0 , 1 ]$ be
	some fixed real numbers satisfying $ \beta + (k-1) \alpha = 1 $.
	Does there exist a graph $G$ admitting some FAT $k$-coloring
	with the corresponding parameters $\alpha$ and $\beta$?
\end{question}

\medskip

The aim of this section is to provide an answer to Question \ref{Q7}. To that end, we first note the following two remarks.

\medskip

\begin{remark} \label{S1}
	Let $\bar{K}_{k}$, the complement of $K_k$, be a graph on vertex set $V \left( \bar{K}_{k} \right) := \left\{ x_1 , x_2 , \dots , x_{k} \right\}$
	whose edge set equals $E \left( \bar{K}_{k} \right) := \varnothing$.
	Let us regard the coloring of $\bar{K}_{k}$ with exactly $k$ colors whose color classes are
	the singletons $V_1 := \left\{ x_1 \right\}$, $V_2 := \left\{ x_2 \right\}$, $\dots$, $V_k := \left\{ x_{k} \right\}$.
	Since $E \left( \bar{K}_{k} \right) := \varnothing$,
	for each vertex $x_i$ and every color class $V_j$ we have
	$
	e \left( x_i, V_j \right) = 0$. Besides, since $\deg \left( x_i \right) = 0$ for all $i = 1 , 2 , \dots , k$,
	we observe that every arbitrary $\alpha , \beta \in [0,1]$ with $ \beta + (k-1) \alpha = 1 $ satisfy
	$$ \alpha \deg \left(x_i\right) =	\beta  \deg \left(x_i\right) = 0 .$$
	Thus, the following equalities hold:
	  $$
	  e \left( x_i, V_j \right) =
	  \begin{cases}
	  	\alpha \deg \left(x_i\right)  &  \mbox{ if }  \ \   x_i \notin V_j \\
	  	\beta  \deg \left(x_i\right)  &  \mbox{ if }  \ \   x_i \in V_j . 
	  \end{cases}
	  $$
	  We conclude that for every arbitrary $\alpha , \beta \in [0,1]$ with $ \beta + (k-1) \alpha = 1 $,
	  the edgeless graph $\bar{K}_{k}$ admits some FAT $k$-coloring with associated parameters $\alpha$ and $\beta$.
\end{remark}

\medskip

\begin{remark} \label{S2}
	In light of Remark \ref{S1}, we henceforth restrict the graphs considered in Question \ref{Q7} to those with no isolated vertices (vertices of degree zero).
	Let $k$, $\alpha$, and $\beta$ be some fixed real numbers such that
	$$\begin{array}{lcccccr}
		k \in \mathbb{N} \setminus  \{1\} , &  &  \alpha , \beta \in [0,1] , &  &  {\rm and}  &  &  \beta + (k-1) \alpha = 1 ,
	\end{array}$$
	 and also let $G$ be a graph with $\delta (G) > 0$ admitting some FAT $k$-coloring with corresponding parameters $\alpha$ and $\beta$ whose nonempty
	 color classes are $V_1$, $V_2$, $\dots$, $V_k$. Now, applying the defining condition of FAT colorings, namely
	 $$
	 e \left( v, V_i \right) =
	 \begin{cases}
	 	\alpha \deg (v)  &  \mbox{ if }  \ \   v \notin V_i \\
	 	\beta \deg (v)   &  \mbox{ if }  \ \   v \in V_i , 
	 \end{cases}
	 $$
	 to the color class $V_1$ and to some vertices $u$ and $w$ with $u \notin V_1$ and $w \in V_1$, yields
	 $$\begin{array}{lcccr}
	 	\alpha = \dfrac{e \left( u , V_1 \right)}{\deg (u)} \in \mathbb{Q} &   &   {\rm and}  &  &  \beta = \dfrac{e \left( w , V_1 \right)}{\deg (w)} \in \mathbb{Q} .
	 \end{array}$$
\end{remark}

\bigskip

On the basis of Remarks \ref{S1} and \ref{S2}, Question \ref{Q7} may be restated as follows.

\medskip

\begin{question} (Restatement of Question \ref{Q7})  \label{Q8} \\
	Let $k > 1$ be a fixed integer, and let $\alpha \in \left[ 0 , \frac{1}{k-1} \right) \cap \mathbb{Q}$
	and $\beta \in [ 0 , 1 ] \cap \mathbb{Q}$ be
	some fixed rational numbers satisfying $ \beta + (k-1) \alpha = 1 $.
	Does there exist a graph $G$ with $\delta (G) > 0$ admitting some FAT $k$-coloring
	with the corresponding parameters $\alpha$ and $\beta$?
\end{question}

\medskip
The following theorem settles Question \ref{Q8} in the affirmative.
Before proceeding to the theorem, it is requisite to introduce some definitions.
A {\it homomorphism} from a graph $G$ to a graph $H$, is defined as a mapping $f \colon V(G) \to V(H)$ such that
the implication
$$xy \in E(G) \ \Longrightarrow \ f(x)f(y) \in E(H)$$
holds for every pair of vertices $x$ and $y$ of $G$.
The notation $G \longrightarrow H$ signifies the existence of a homomorphism from $G$ to $H$.
We say that two graphs $G$ and $H$ are {\it homomorphically equivalent} precisely when both $G \longrightarrow H$ and $H \longrightarrow G$
are valid. Homomorphically equivalent graphs are known to enjoy a considerable variety of common characteristics.
For instance, homomorphically equivalent graphs share the same clique number as well as the same chromatic number.
An extensive treatment of topics related to graph homomorphisms can be found in \cite{MR2089014}.

At this stage, we are in a position to formally state and establish the main theorem of the paper.

\medskip

\begin{theorem} \label{positive}
	Let $k > 1$ be a fixed integer, and let $\alpha \in \left[0, \tfrac{1}{k-1}\right) \cap \mathbb{Q}$ 
	and $\beta \in [0,1] \cap \mathbb{Q}$ be fixed rational numbers satisfying 
	$\beta + (k-1)\alpha = 1$. Then there exists a sequence 
	$\displaystyle\left\{G_n\right\}_{n=1}^{\infty}$ of pairwise non-homomorphically equivalent graphs, 
	each being a regular graph of positive degree, admitting a FAT $k$-coloring with the corresponding 
	parameters $\alpha$ and $\beta$.
\end{theorem}

\begin{proof}
    If $\alpha = 0$, then for each fixed positive integer $n$ let $H^{1}_n$, $H^{2}_n$, $\dots$, $H^{k}_n$ be $k$ vertex disjoint complete graphs,
	each of order $n$. We define $G_n$ to be their union.
	More formally, define the graph $G_n$ with vertex set
	$$ V \left( G_n \right) := V \left( H^{1}_n \right) \cup  V \left( H^{2}_n \right) \cup \dots \cup  V \left( H^{k}_n \right) $$
	whose edge set is
	$$ E \left( G_n \right) := E \left( H^{1}_n \right) \cup  E \left( H^{2}_n \right) \cup \dots \cup  E \left( H^{k}_n \right) . $$
	Now, the sets
	$V\left( H^{1}_n \right)$, $V\left( H^{2}_n \right)$, $\dots$, $V\left( H^{k}_n \right)$ form the color classes of a
	FAT $k$-coloring of $G_n$ with associated parameters $\alpha = 0$ and $\beta = 1$. Moreover, since the clique number of $G_n$
	equals $n$, the graphs $G_i$ and $G_j$ have distinct clique numbers for every pair of distinct positive integers $i$ and $j$;
	hence no two such graphs in the sequence $\displaystyle\left\{G_n\right\}_{n=1}^{\infty}$ are homomorphically equivalent.
	
	For the remainder of the proof, we shall assume that $\alpha > 0$.
	Since $\alpha$ is rational, there exist positive integers $a_{\scriptscriptstyle 0}$ and $b_{\scriptscriptstyle 0}$
	such that	$\alpha = \frac{a_{\scriptscriptstyle 0}}{b_{\scriptscriptstyle 0}}$.
	Moreover, 
	since $0 < \alpha < \frac{1}{k-1}$, we have $b_{\scriptscriptstyle 0} - (k-1) a_{\scriptscriptstyle 0} \in \mathbb{N}$.
	Hence, there exists a positive integer $g$ such that
	$ g \bigl( b_{\scriptscriptstyle 0} - (k-1) a_{\scriptscriptstyle 0} \bigr) > k $. We regard such a positive integer $g$ as fixed.
	
	We now define the sequence $\displaystyle\left\{G_n\right\}_{n=1}^{\infty}$. We proceed to give the general term $G_n$. 
	Fix an arbitrary positive integer $n$; and set
	$$\begin{array}{lcccr}
		a:= (g+n)a_{\scriptscriptstyle 0} &  {\rm and} &  b:= (g+n)b_{\scriptscriptstyle 0} &  {\rm and} &  \ell := b - (k-1) a + 1 . 
	\end{array}$$
	Now, we have $\alpha = \frac{a}{b}$. Also, the positive integer $\ell$ is strictly greater than $k$.
	
    \noindent
	Let $G_n$ be the graph whose vertex set equals the Cartesian product
	$$V(G) := [k] \times [a] \times [\ell] ;$$
	and two vertices $\left( r , s , t \right)$ and $\left( r' , s' , t' \right)$
	are declared to be adjacent in $G_n$ if and only if one of the following two conditions holds:
	$$\begin{array}{llcl}
		  \bullet & \mbox{The first condition:}  &  &  \boxed{t=t' \mbox{ and } r \neq r'}                     \\
		          &                              &  &                                                          \\
		  \bullet & \mbox{The second condition:} &  &  \boxed{t \neq t' \mbox{ and } r=r' \mbox{ and } s=s' .}
	\end{array}$$
	For every fixed $t \in [ \ell ]$, the induced subgraph of $G_n$ on $[k] \times [a] \times \{t\}$ is a complete
	$k$-partite graph with all parts of cardinality $a$.
    Indeed, the induced subgraph of $G_n$ on $[k] \times [a] \times \{t\}$ is a complete $k$-partite graph
    whose $k$ parts are $\{i\} \times [a] \times \{t\}$ for all $i = 1,2,\dots , k$.
    Also, for every fixed $r \in [k]$ and each fixed $s \in [a]$, the induced subgraph of $G_n$ on
    $\{r\} \times \{s\} \times [\ell]$ is a complete graph $K_{\ell}$.
    Besides, each neighbor of $(r,s,t)$ lies within
    $\bigl(  [k] \times [a] \times \{t\}  \bigr) \cup \bigl(  \{r\} \times \{s\} \times [\ell]  \bigr)$.
    Moreover, the neighborhood of the vertex $(r,s,t)$ is $A \cup B$, where
    $$\begin{array}{llcrr}
    	A := \bigl(  [k] \setminus \{r\}  \bigr) \times [a] \times \{t\} &  &  {\rm and}  &  &  B := \{r\} \times \{s\} \times \bigl(  [\ell] \setminus \{t\}  \bigr) .
    \end{array}$$
    Since $A \cap B = \varnothing$, the degree of the vertex $(r,s,t)$ equals $(k-1)a + ( \ell - 1 )$; which is equal to $b$.
    We conclude that $G_n$ is a $b$-regular graph.
    
    \noindent
    If we color every vertex $(r,s,t)$ of $G_n$ with its first coordinate $r$,
    then $V_1$, $V_2$, $\dots$, $V_k$ would be the color classes of this coloring, where
    $$\begin{array}{lcr}
    	V_i := \{i\} \times [a] \times [\ell]     &  \mbox{for all}  &    i = 1 , 2 , \dots , k .
    \end{array}$$
    Now, for each vertex $(r,s,t)$ of $G_n$, the set of its neighbors that are included in $V_r$ is
    $$ V_r \cap N (r,s,t) = \{r\} \times \{s\} \times \bigl(  [\ell] \setminus \{t\}  \bigr) ;$$
    which implies
    $$ \bigl|  V_r \cap N (r,s,t)  \bigr| = \ell - 1 = b - (k-1)a .$$
    Also, the set of all neighbors of $(r,s,t)$ that are included in $V_i$ where $i \neq r$, is equal to
    $$ V_i \cap N (r,s,t) = \{i\} \times [a] \times \{t\} ;$$
    and therefore, the condition $i \neq r$ implies
    $$ \bigl|  V_i \cap N (r,s,t)  \bigr| = a .$$
    Accordingly, for each vertex $(r,s,t)$ and every color class $V_i$ we have:
    $$ \bigl|  V_i \cap N (r,s,t)  \bigr| =
    \begin{cases}
    	a            &  \mbox{  if }  \ \   (r,s,t) \notin V_i \\
    	b - (k-1)a   &  \mbox{  if }  \ \   (r,s,t) \in V_i . 
    \end{cases}$$
    Now, since $G_n$ is $b$-regular and noting that
    $$\begin{array}{lcr}
    	\boxed{a = \dfrac{a}{b} \cdot b = \alpha \cdot b}    &  {\rm and}  &   
    	\boxed{b - (k-1)a = \biggl( 1 - (k-1) \dfrac{a}{b} \biggr) \cdot b = \beta \cdot b} \\
    \end{array}$$
    we conclude that for each vertex $(r,s,t)$ and every color class $V_i$, the following conditions
    of FAT colorings are satisfied:
    $$ \bigl|  V_i \cap N (r,s,t)  \bigr| =
    \begin{cases}
    	\alpha \deg (r,s,t)            &  \mbox{  if }  \ \   (r,s,t) \notin V_i \\
    	\beta  \deg (r,s,t)            &  \mbox{  if }  \ \   (r,s,t) \in V_i . 
    \end{cases}$$
    Thus, a $b$-regular graph $G_n$ together with some appropriate FAT $k$-coloring of $G_n$ realizing the given corresponding parameters
    $\alpha$ and $\beta$ is constructed.
    
    \noindent
    The clique number of the constructed graph $G_n$ equals
    $$\omega \left( G_n \right) = \max \{k,\ell\}.$$
    Since $\ell > k$, it follows that
    $$\omega \left( G_n \right) = (g+n)\bigl( b_{\scriptscriptstyle 0} - (k-1) a_{\scriptscriptstyle 0} \bigr) + 1 .$$
    Thus, $\omega \left( G_i \right) \neq \omega \left( G_j \right) $ whenever $i\neq j$. Consequently, no two graphs in the sequence
    $\displaystyle\left\{G_n\right\}_{n=1}^{\infty}$ are homomorphically equivalent, which completes the proof.
\end{proof}

\vskip 0.4 true cm

\bibliographystyle{amsplain}
\def\cprime{$'$} \def\cprime{$'$}
\providecommand{\bysame}{\leavevmode\hbox to3em{\hrulefill}\thinspace}
\providecommand{\MR}{\relax\ifhmode\unskip\space\fi MR }
\providecommand{\MRhref}[2]{%
	\href{http://www.ams.org/mathscinet-getitem?mr=#1}{#2}
}
\providecommand{\href}[2]{#2}



\bigskip
\bigskip


{\footnotesize {\bf Saeed Shaebani}\; \\ {School of Mathematics and Computer Science}, {Damghan University,} {Damghan, Iran.}\\
{\tt Email: shaebani@du.ac.ir}\\

\end{document}